\documentclass[a4paper]{amsart}

\usepackage[utf8]{inputenc}
\usepackage[T1]{fontenc}
\usepackage{lmodern, enumerate}
\usepackage{amssymb,amsxtra,amscd,amsmath}
\usepackage{amsthm}
\usepackage{stmaryrd}
\usepackage{amsfonts}
\usepackage{tikz-cd}
\usepackage{bbm}
\usepackage[all]{xy}
\usepackage{xcolor}
\usepackage{nicefrac,mathtools}
\usepackage[inline]{enumitem}
\usepackage{microtype}
\usepackage{comment}
\usepackage[
  pdftitle={On the relation between the product of KK-groups and the KK-group of the product},
  pdfauthor={Diego Mart\'{i}nez},
  pdfsubject={Mathematics},
  colorlinks=true, linkcolor=blue, linkbordercolor=blue, citecolor=red, citebordercolor=red, urlcolor=blue, linktocpage=true
]{hyperref}
\usepackage[capitalise, nameinlink, noabbrev, nosort]{cleveref} % see the cleveref documentation to find out what these options do
\usepackage[lite]{amsrefs}

\newtheorem{alphthm}{Theorem}     %letter numbering

\newtheorem{alphcor}[alphthm]{Corollary}           %letter numbering

\setlist[enumerate]{font=\normalfont}
\crefname{enumi}{}{} % no name for list items
\crefname{enumi}{}{} % no name for nested list items
\setlist[enumerate]{label=(\roman*)}

\numberwithin{equation}{section}
\theoremstyle{plain}

\newtheorem{lemma}[equation]{Lemma}
\newtheorem{proposition}[equation]{Proposition}
\newtheorem{corollary}[equation]{Corollary}

\theoremstyle{definition}

\theoremstyle{remark}

\crefname{theorem}{Theorem}{Theorems}
\crefname{alphthm}{Theorem}{Theorems}
\crefname{proposition}{Proposition}{Propositions}
\crefname{lemma}{Lemma}{Lemmas}
\crefname{claim}{Claim}{Claims}
\crefname{remark}{Remark}{Remarks}
\crefname{conjecture}{Conjecture}{Conjectures}
\crefname{corollary}{Corollary}{Corollaries}
\crefname{question}{Question}{Questions}
\crefname{conjecture}{Conjecture}{Conjectures}
\crefname{fact}{Fact}{Facts}
\crefname{claim}{Claim}{Claims}
\crefname{case}{Case}{Cases}
\crefname{convention}{Convention}{Conventions}

\calclayout

\NewDocumentCommand{\CC}{}{\mathbb{C}}

             \NewDocumentCommand{\NN}{}{\mathbb{N}}

             \NewDocumentCommand{\ZZ}{}{\mathbb{Z}}

\NewDocumentCommand{\CA}{}{\mathcal{A}}            
\NewDocumentCommand{\CCC}{}{\mathcal{C}}           
            \NewDocumentCommand{\CF}{}{\mathcal{F}}
            \NewDocumentCommand{\CH}{}{\mathcal{H}}
            
\NewDocumentCommand{\CK}{}{\mathcal{K}}            
\NewDocumentCommand{\CM}{}{\mathcal{M}}            
\NewDocumentCommand{\CO}{}{\mathcal{O}}

\NewDocumentCommand{\CU}{}{\mathcal{U}}

\NewDocumentCommand{\fs}{}{\mathfrak{s}}

\NewDocumentCommand{\Egrp}{omm}{%
  \IfNoValueTF{#1}
    {E\left(#2,#3\right)}
    {E_{#1}\left(#2,#3\right)}
}
\NewDocumentCommand{\Etop}{omm}{%
  \IfNoValueTF{#1}
    {E^{\rm top}_{\star}\left(#2; #3\right)}
    {E^{\rm top}_{\star,#1}\left(#2; #3\right)}
}
\NewDocumentCommand{\Kth}{}{K}
\NewDocumentCommand{\KKth}{}{KK}

\NewDocumentCommand{\KKgrp}{omm}{%
  \IfNoValueTF{#1}
    {KK\left(#2,#3\right)}
    {KK_{#1}\left(#2,#3\right)}
}
\NewDocumentCommand{\KLgrp}{omm}{%
  \IfNoValueTF{#1}
    {KL\left(#2,#3\right)}
    {KL_{#1}\left(#2,#3\right)}
}
\NewDocumentCommand{\Kgrp}{om}{%
  \IfNoValueTF{#1}
    {K_\star\left(#2\right)}
    {K_{#1}\left(#2\right)}
}
\NewDocumentCommand{\totKgrp}{m}{\underline{\textbf{K}}\left(#1\right)}
\NewDocumentCommand{\Extgrp}{omm}{%
  \IfNoValueTF{#1}
    {{\rm Ext}_{\ZZ}^1\left(#2,#3\right)}
    {{\rm Ext}_{#1}^1\left(#2,#3\right)}
}
\NewDocumentCommand{\Homgrp}{omm}{%
  \IfNoValueTF{#1}
    {{\rm Hom}\left(#2,#3\right)}
    {{\rm Hom}_{#1}\left(#2,#3\right)}
}

\NewDocumentCommand{\variable}{}{-}

\NewDocumentCommand{\hkindalg}{om}{%
  \IfNoValueTF{#1}
    {\CCC_0\left(#2\right)}
    {\CCC_{0,#1}\left(#2\right)}
}
\NewDocumentCommand{\hkalg}{om}{%
  \IfNoValueTF{#1}
    {\CA\left(#2\right)}
    {\CA_{ {#1} }\left(#2\right)}
}

\NewDocumentCommand{\hkhilb}{om}{%
  \IfNoValueTF{#1}
    {\CH(#2)}
    {\CH_{#1}(#2)}
}

\NewDocumentCommand{\fhalg}{O{h}m}{%
  \IfNoValueTF{#1}
    {\CF\left(#2\right)}
    {\CF_{#1}\left(#2\right)}
}
\NewDocumentCommand{\schwartzfuncs}{om}{%
  \IfNoValueTF{#1}
    {\fs\left(#2\right)}
    {\fs\left(#1,#2\right)}
}
\NewDocumentCommand{\delmaux}{}{\alpha}
\NewDocumentCommand{\ddelmaux}{}{\beta}
\NewDocumentCommand{\delm}{o}{%
  \IfNoValueTF{#1}
    {\delmaux}
    {\delmaux^{#1}}
}
\NewDocumentCommand{\reddelm}{o}{%
  \IfNoValueTF{#1}
    {\delmaux_{\red}}
    {\delmaux^{#1}_{\red}}
}
\NewDocumentCommand{\maxdelm}{o}{%
  \IfNoValueTF{#1}
    {\delmaux_{{\rm max}}}
    {\delmaux^{#1}_{{\rm max}}}
}
\NewDocumentCommand{\ddelm}{o}{%
  \IfNoValueTF{#1}
    {\ddelmaux}
    {\ddelmaux^{#1}}
}
\NewDocumentCommand{\redddelm}{o}{%
  \IfNoValueTF{#1}
    {\ddelmaux_{\red}}
    {\ddelmaux^{#1}_{\red}}
}
\NewDocumentCommand{\maxddelm}{o}{%
  \IfNoValueTF{#1}
    {\ddelmaux_{{\rm max}}}
    {\ddelmaux^{#1}_{{\rm max}}}
}

\NewDocumentCommand{\crossedprod}{omm}{%
  \IfNoValueTF{#1}
    {#2 \rtimes #3}
    {#2 \rtimes_{#1} #3}
}
\NewDocumentCommand{\algcrossedprod}{omm}{%
  \IfNoValueTF{#1}
    {#2 \rtimes_{{\rm alg}} #3}
    {#2 \rtimes^{#1}_{{\rm alg}} #3}
}
\NewDocumentCommand{\redcrossedprod}{omm}{%
  \IfNoValueTF{#1}
    {#2 \rtimes_{\red} #3}
    {#2 \rtimes^{#1}_{\red} #3}
}
\NewDocumentCommand{\maxcrossedprod}{omm}{%
  \IfNoValueTF{#1}
    {#2 \rtimes_{\full} #3}
    {#2 \rtimes^{#1}_{\full} #3}
}

\NewDocumentCommand{\cstar}{}{\texorpdfstring{\(C^*\)\nobreakdash-\hspace{0pt}}{*-}}
\NewDocumentCommand{\Star}{}{\texorpdfstring{\(^*\)\nobreakdash-\hspace{0pt}}{*-}}

\NewDocumentCommand{\range}{o}{%
  \IfNoValueTF{#1}
    { {\rm rng} }
    { {\rm rng} \left(#1\right) }
}
\NewDocumentCommand{\source}{o}{%
  \IfNoValueTF{#1}
    { {\rm source} }
    { {\rm source} \left(#1\right) }
}

\NewDocumentCommand{\innerprod}{mm}{\langle #1, #2 \rangle}
\NewDocumentCommand{\linnerprod}{omm}{%
  \IfNoValueTF{#1}
    {_{\rm L}\innerprod{#2}{#3}}
    {_{#1}\innerprod{#2}{#3}}
}
\NewDocumentCommand{\rinnerprod}{omm}{%
  \IfNoValueTF{#1}
    {\innerprod{#2}{#3}_{\rm R}}
    {\innerprod{#2}{#3}_{#1}}
}

\NewDocumentCommand{\multialg}{m}{\CM\left(#1\right)}

\NewDocumentCommand{\red}{}{\rm red}
\NewDocumentCommand{\full}{}{\rm max}

\NewDocumentCommand{\alg}{}{\rm alg}
\NewDocumentCommand{\redalg}{om}{%
  \IfNoValueTF{#1}
    {C^*_{\red}\left(#2\right)}
    {C^*_{\red,#1}\left(#2\right)}
}
\NewDocumentCommand{\redalgnopar}{om}{%
  \IfNoValueTF{#1}
    {C^*_{\red}(#2)}
    {C^*_{\red,#1}(#2)}
}
\NewDocumentCommand{\maxalg}{om}{%
  \IfNoValueTF{#1}
    {C^*_{\full}\left(#2\right)}
    {C^*_{\full,#1}\left(#2\right)}
}             % Maximal
\NewDocumentCommand{\maxalgnopar}{om}{%
  \IfNoValueTF{#1}
    {C^*_{\full}(#2)}
    {C^*_{\full,#1}(#2)}
}

\NewDocumentCommand{\algalg}{om}{%
  \IfNoValueTF{#1}
    {\CC_{\alg}\left(#2\right)}
    {\CC_{\alg,#1}\left(#2\right)}
}
\NewDocumentCommand{\algalgnopar}{om}{%
  \IfNoValueTF{#1}
    {\CC_{\alg}(#2)}
    {\CC_{\alg,#1}(#2)}
}

\NewDocumentCommand{\singideal}{om}{%
  \IfNoValueTF{#1}
    {J_{#2}}
    {J_{#2}^{#1}}
}
\NewDocumentCommand{\condexp}{o}{%
  \IfNoValueTF{#1}
    {E}
    {E_{#1}}
}
\NewDocumentCommand{\redcondexp}{o}{%
  \IfNoValueTF{#1}
    {P}
    {P_{#1}}/
}
\NewDocumentCommand{\esscondexp}{o}{%
  \IfNoValueTF{#1}
    {EL}
    {EL_{#1}}
}

\NewDocumentCommand{\tensor}{}{\otimes}

\NewDocumentCommand{\id}{o}{%
  \IfNoValueTF{#1}
    {{\rm id}}
    {{\rm id}_{#1}}
}

\begin{document}

%%%%%%%%%%%%%%%%%%%%%%%%%%%%%%%%%%%%%%%%%%%%%%%%%%%%%%%%%%%%%%%%%%%%%%%%%%
% TITLE AND DATA
%%%%%%%%%%%%%%%%%%%%%%%%%%%%%%%%%%%%%%%%%%%%%%%%%%%%%%%%%%%%%%%%%%%%%%%%%%
\title[KK-group of product vs. product of KK-groups]{On the relation between the product of KK-groups and the KK-group of the product}
\author[Diego Mart\'{i}nez]{Diego Mart\'{i}nez \(^{1}\)}
\address{Department of Mathematics, KU Leuven, Celestijnenlaan 200B, 3001 Leuven, Belgium.}
\email{diego.martinez@kuleuven.be}

\begin{abstract}
  We observe that the canonical map \(\KKgrp{A}{\prod_{n \in \NN} B_n} \to \prod_{n \in \NN} \KKgrp{A}{B_n}\) is an isomorphism of abelian groups whenever \(A\) enjoys the Universal Coefficient Theorem and \(B_n\) are unital, simple and purely infinite \cstar{}algebras.
  This clarifies an aspect of previous work of Dadarlat--Eilers and Tikuisis--White--Winter.
\end{abstract}

% 19K35 : KK theory
% 46L80 : K-theory and operator algebras


\thanks{{\(^{1}\)} Supported by projects G085020N and 1218726N funded by the Research Foundation Flanders (FWO)}

%%%%%%%%%%%%%%%%%%%%%%%%%%%%%%%%%%%%%%%%%%%%%%%%%%%%%%%%%%%%%%%%%%%%%%%
% TITLE
%%%%%%%%%%%%%%%%%%%%%%%%%%%%%%%%%%%%%%%%%%%%%%%%%%%%%%%%%%%%%%%%%%%%%%%
\maketitle
%\tableofcontents

%%%%%%%%%%%%%%%%%%%%%%%%%%%%%%%%%%%%%%%%%%%%%%%%%%%%%%%%%%%%%%%%%%%%%%%
% INTRODUCTION
%%%%%%%%%%%%%%%%%%%%%%%%%%%%%%%%%%%%%%%%%%%%%%%%%%%%%%%%%%%%%%%%%%%%%%%
\section{Introduction} \label{sec:intro}

Inspired by the pioneering work of Brown, Douglas and Fillmore \cite{brown-douglas-fillmore-1977}, Rosenberg and Schochet \cite{rosenberg-schochet-1987-uct} proved that for a large class of separable \cstar{}algebras there is a short exact sequence
\begin{equation} \label{eq:intro-uct}
  0 \to \Extgrp{\Kgrp[\star]{A}}{\Kgrp[\star + 1]{B}} \to \KKgrp{A}{B} \to \Homgrp{\Kgrp[\star]{A}}{\Kgrp[\star]{B}} \to 0.
\end{equation}
If \(A\) has the property that \eqref{eq:intro-uct} is short exact for all separable \(B\) then \(A\) is said to \emph{enjoy}, or \emph{satisfy}, the \emph{Universal Coefficient Theorem} (or \emph{UCT} for short).
The importance of this property cannot be overstated in the study of nuclear (separable) \cstar{}algebras, as it allows to completely describe the, a priori non-approachable, group \(\KKgrp{A}{B}\) by the vastly easier to compute groups \(\Kgrp[\star]{A}\) and \(\Kgrp[\star]{B}\).
We refer the reader to \cites{TikuisisAaron2017QonC,dadarlat-eilers-2002,Kasparov-1988,blackadar-book-k-theory,rosenberg-schochet-1987-uct,higson-kasparov-2000} for several deep applications.

Particularly, one of the reasons why \eqref{eq:intro-uct} is so relevant in the current research is its use in the classification of \cstar{}algebras.
In their groundbreaking \cite{TikuisisAaron2017QonC}; and partly based on previous ideas of Dadarlat and Eilers \cite{dadarlat-eilers-2002}; Tikuisis, White and Winter \cite{TikuisisAaron2017QonC}*{Section 3} are interested in the injectivity of the map 
\[
  \Phi_{nuc} \colon \KKgrp[nuc]{A}{\prod_{n \in \NN} B_n} \to \prod_{n \in \NN} \KKgrp[nuc]{A}{B_n},
\]
where \(A\) and \(\{B_n\}_{n \in \NN}\) are separable \cstar{}algebras, see \cite{TikuisisAaron2017QonC}*{Equation (3.2)}.
They, in fact, need something slightly different: as Dadarlat and Eilers in \cite{dadarlat-eilers-2002}*{Theorem 4.10}, they need the map
\(
  \KKgrp[nuc]{A}{\prod_{n \in \NN} B_n} \to \prod_{n \in \NN} \Homgrp[\Lambda]{\totKgrp{A}}{\totKgrp{B_n}}
\)
to be injective whenever the \cstar{}algebras \(\{B_n\}_{n \in \NN}\) are sufficiently well-behaved.
It is suggested in \cite{TikuisisAaron2017QonC} that for injectivity of \(\Phi_{nuc}\) above one needs \(\{B_n\}_{n \in \NN}\) to enjoy some mild properties.
In this brief note we clarify this aspect.

\begin{alphthm} \label{main-theorem}
  Let \(A, \{B_n\}_{n \in \NN}\) be separable \cstar{}algebras, where \(A\) satisfies the UCT.
  Suppose that either the canonical map \(\phi \colon \Kgrp[\star]{\prod_{n \in \NN} B_n} \to \prod_{n \in \NN} \Kgrp[\star]{B_n}\) is injective and \(\Kgrp[\star]{A}\) is free, or that \(\phi\) is an isomorphism.
  Then
  \begin{align*}
    \Phi \colon \KKgrp{A}{\prod_{n \in \NN} B_n} \to \prod_{n \in \NN} \KKgrp{A}{B_n}
  \end{align*}
  is an injective group homomorphism.
  In particular, this is the case whenever:
  \begin{enumerate}
    \item \(\{B_n\}_{n \in \NN}\) are simple, unital, and purely infinite; or
    \item \(\{B_n\}_{n \in \NN}\) are simple, unital, infinite dimensional and tracially AF, and \(\Kgrp[\star]{A}\) is free.
  \end{enumerate}
\end{alphthm}

\begin{alphcor} \label[corollary]{intro-cor}
  If \(A\) is separable and enjoys the UCT, and \(\{B_n\}_{n \in \NN}\) are separable, unital, simple and purely infinite, then \(\KKgrp{A}{\prod_{n \in \NN} B_n} \cong \prod_{n \in \NN} \KKgrp{A}{B_n}\).
\end{alphcor}

\noindent \emph{Remark.} Similar results to \cref{intro-cor} follow from work Gabe \cite{gabe-memoirs-2024}. These will be discussed in \cref{cor-kl}.

\vspace{1.5mm}

Morally speaking, \cref{main-theorem} states that whenever \(A\) enjoys the UCT, then countably-many \Star{}homomorphisms are null-homotopic if and only if they are \emph{uniformly} null-homotopic.
This reinforces the idea that the UCT is a sort of ``Baire Category Theorem'' for \cstar{}algebras: it allows to perform compactness arguments on Kasparov's bivariant \KKth-theory \cite{Kasparov-1988}.
This is, of course, modulo the same result holding for the usual \Kth-theory of \(\{B_n\}_{n \in \NN}\), which does \emph{not} necessarily hold.
Thus explaining the extra assumption in \cref{main-theorem}.

\vspace{1.5mm}

\noindent \textbf{Acknowledgements:} we would like to thank G\'abor Szab\'o, James Gabe and an anonymous referee for some comments about a previous version of this manuscript.

\section{Proof of \texorpdfstring{\cref{main-theorem}}{Theorem A}}

To begin the discussion, recall that \(\Kgrp[\star]{B} \coloneqq \Kgrp[0]{B} \oplus \Kgrp[1]{B}\) and that, similarly, \(\Kgrp[\star+1]{B} \coloneqq \Kgrp[1]{B} \oplus \Kgrp[0]{B}\).
Moreover, group homomorphisms \(\Kgrp[\star]{A} \to \Kgrp[\star]{B}\) are \emph{graded}, so they are a sum of two group homomorphisms \(\Kgrp[0]{A} \to \Kgrp[0]{B}\) and \(\Kgrp[1]{A} \to \Kgrp[1]{B}\).
There is a canonical map
\begin{equation} \label{phi}
  \phi \colon \Kgrp[\star]{\prod_{n \in \NN} B_n} \to \prod_{n \in \NN} \Kgrp[\star]{B_n},
\end{equation}
which is the toy version of the map \(\Phi\) in \cref{main-theorem}.
Quoting Dadarlat and Eilers~\cite{dadarlat-eilers-2002}*{Section 3.2}: ``\emph{it is well known, although perhaps not as well known as it should be, that \(\phi\) is not an isomorphism in general}''.
Notice that, in grade \(0\), \(\phi\) sends the class \([(p_n)_{n \in \NN}]_0\) of a tuple of projections to the tuple \(([p_n]_0)_{n \in \NN}\).
Whence, in general, \(p_n \in {\rm Proj}(M_{m(n)}(B_n))\) and there is no uniform bound on \(m(n)\).\footnote{\, ``Uniform'' here means among the elements of \(\Kgrp[0]{\prod_{n \in \NN} B_n}\). Of course there is no uniform bound among the sizes of the projections one may consider.}
Likewise there is no bound in the dimension witnessing the equality \([(p_n)_{n \in \NN}]_0 = [(q_n)_{n \in \NN}]_0\).

\vspace{1.5mm}

Before proceeding with the proof we need some well known lemmas from homological algebra, whose proofs we sketch only for the \cstar{}reader.
Henceforth, \(\Gamma\) and \(\{\Lambda_n\}_{n \in \NN}\) are discrete abelian groups.

\begin{lemma} \label[lemma]{hom-iso}
  There is a group isomorphism \(\Psi \colon \Homgrp{\Gamma}{\prod_{n \in \NN} \Lambda_n} \to \prod_{n \in \NN} \Homgrp{\Gamma}{\Lambda_n}\).
\end{lemma}
\begin{proof}
  An element \(\varphi \in \Homgrp{\Gamma}{\prod_{n \in \NN} \Lambda_n}\) is just given by countably many maps \(\varphi = (\varphi_n)_{n \in \NN}\).
  Sending \(\varphi \mapsto (\varphi_n)_{n \in \NN}\) is, morally, the identity, and it is an injective group homomorphism. 
  The fact that it also is surjective follows since any countably choice of maps \(\{\varphi_n\}_{n \in \NN}\) may be tupled up to \(\varphi\).
\end{proof}

\begin{lemma} \label[lemma]{ext-inj}
  There is a group isomorphism
  \(\Xi \colon \Extgrp{\Gamma}{\prod_{n \in \NN} \Lambda_n} \to \prod_{n \in \NN} \Extgrp{\Gamma}{\Lambda_n}\).
\end{lemma}
\begin{proof}
  Let \(0 \to F_1 \to F_0 \to \Gamma \to 0\) be a short exact sequence, where \(F_0, F_1\) are free abelian groups. Denote by \(f\) the map \(f \colon F_1 \to F_0\) appearing in the short exact sequence.
  Pre-composition by \(f\) defines a group homomorphism \(f^* \colon \Homgrp{F_0}{\Lambda} \to \Homgrp{F_1}{\Lambda}\) for any abelian group \(\Lambda\).
  Moreover, it is well known from homological algebra that
  \[
    \Extgrp{\Gamma}{\Lambda} \cong {\rm coker}\left(\Homgrp{F_0}{\Lambda} \xrightarrow{f^*} \Homgrp{F_1}{\Lambda}\right).
  \]
  With the \cstar{}inclined reader in mind: this may be shown directly.
  Indeed, given any element \(\eta \colon 0 \to \Lambda \to E \to \Gamma \to 0\) in \(\Extgrp{\Gamma}{\Lambda}\) one may lift the quotient \(F_0 \to \Gamma\) to a map \(F_0 \to E\), since \(F_0\) is free (and abelian) by assumption.
  Restricting the lift \(F_0 \to E\) to \(F_1 = \ker(F_0 \to \Gamma)\) yields a group homomorphism \(\gamma(\eta) \colon F_1 \to \Lambda\).
  Now, different choices of lift \(F_0 \to E\) change \(\gamma(\eta)\) by an element in \(\Homgrp{F_0}{\Lambda}\), which in particular precisely means that the assignment
  \[
    \Extgrp{\Gamma}{\Lambda} \ni \left[\eta\right] \mapsto \left[\gamma\left(\eta\right)\right] \in {\rm coker}\left(\Homgrp{F_0}{\Lambda} \xrightarrow{f^*} \Homgrp{F_1}{\Lambda}\right)
  \]
  is a well-defined group homomorphism.
%  To clarify the situation, and to improve legibility, the setting may be summarized as the following diagram
%  \begin{equation*}
%    \begin{tikzcd}
%      \eta \colon 0 \arrow{r}{} & \Lambda \arrow{r}{} & E \arrow{r}{} & \Gamma \arrow{r}{} \arrow[equal]{d}{} & 0 \\
%      0 \arrow{r}{} & F_1 \arrow{r}{f} \arrow{u}{\gamma\left(\eta\right)} & F_0 \arrow{r}{} \arrow[dashed]{u}{} & \Gamma \arrow{r}{} & 0 
%    \end{tikzcd}
%  \end{equation*}
%  being commutative.
  Moreover, one can show that \([\eta] \mapsto [\gamma(\eta)]\) is an isomorphism.
%  In order to do this, let \(\beta \in \Homgrp{F_1}{\Lambda}\) be any given homomorphism \(F_1 \to \Lambda\).
%  The extension
%  \[
%    \zeta\left(\beta\right) \colon 0 \to \Lambda \to \left(\Lambda \oplus F_0\right)/\left\{\left(\beta\left(x\right), -f\left(x\right)\right) : x \in F_1\right\} \to \Gamma \to 0
%  \]
%  can be seen to satisfy \([\gamma(\zeta(\beta))] = [\beta]\), so that \(\gamma\) has a left inverse.
%  It can also be proven to be a right one.

  Using these ideas for \(\Lambda = \prod_{n \in \NN} \Lambda_n\) and applying \cref{hom-iso}, the map \(\Xi\) in the statement can be rewritten to be the map
  \begin{equation*}
    {\rm coker}\left(\prod_{n \in \NN} \Homgrp{F_0}{\Lambda_n} \xrightarrow{(f^*_n)_n} \prod_{n \in \NN} \Homgrp{F_1}{\Lambda_n}\right) \xrightarrow{\Xi} \prod_{n \in \NN} {\rm coker} \left(\Homgrp{F_0}{\Lambda_n} \xrightarrow{f_n^*} \Homgrp{F_1}{\Lambda_n}\right),
  \end{equation*}
  which sends \([(b_n)_{n \in \NN}]\) to \(([b_n])_{n \in \NN}\), where \(b_n \in \Homgrp{F_1}{\Lambda_n}\).
  As this map is defined coordinate wise it is not hard to check that it is, indeed, injective.
\end{proof}

Note that \cref{ext-inj} will not be sufficient for the proof of \cref{main-theorem}, as it only deals with taking the infinite product \(\prod_{n \in \NN}\) out of the \(\Extgrp{\variable}{\variable}\): it does not deal with taking it out of \(\Kgrp[\star]{\variable}\).
For that we will need an alternative argument.
If \(\phi \colon \Lambda_1 \to \Lambda_2\) is a group homomorphism then the picture of \(\Extgrp{\Gamma}{\variable}\) in the (proof of) \cref{ext-inj} allows to construct an induced group homomorphism \(\phi^* \colon \Extgrp{\Gamma}{\Lambda_1} \to \Extgrp{\Gamma}{\Lambda_2}\) via:
\begin{align} \label{eq-phi-induced}
  \phi^* \colon {\rm coker}\left(\Homgrp{F_0}{\Lambda_1} \xrightarrow{f^*} \Homgrp{F_1}{\Lambda_1}\right) & \to {\rm coker}\left(\Homgrp{F_0}{\Lambda_2} \xrightarrow{f^*} \Homgrp{F_1}{\Lambda_2}\right), \nonumber \\
  \beta + f^*\left(\Homgrp{F_0}{\Lambda_1}\right) & \mapsto \phi \beta + f^*\left(\Homgrp{F_0}{\Lambda_2}\right).
\end{align}

We now turn to the study of the ``in particular'' statement in \cref{main-theorem}.
We are interested in the injectivity (or bijectivity) of the toy map \(\phi \colon \Kgrp[\star]{\prod_{n \in \NN} B_n} \to \prod_{n \in \NN} \Kgrp[\star]{B_n}\) in \eqref{phi}.

\begin{proposition} \label[proposition]{inj-kth}
  If \(\{B_n\}_{n \in \NN}\) are unital, simple, separable, and purely infinite then \(\phi \colon \Kgrp[\star]{\prod_{n \in \NN} B_n} \to \prod_{n \in \NN} \Kgrp[\star]{B_n}\) is an isomorphism of abelian groups.

  Likewise, if \(\{B_n\}_{n \in \NN}\) are unital, simple, infinite dimensional and tracially AF then \(\phi\) is injective. 
\end{proposition}
\begin{proof}
  This follows from \cite{dadarlat-eilers-2002}*{Section 3.2 and Propositions 6.6 and 6.15}.
  Indeed, in \cite{dadarlat-eilers-2002}*{Proposition 6.15} it is sketched why a purely infinite, simple, unital \cstar{}algebra is an \emph{admissible algebra of infinite type}~\cite{dadarlat-eilers-2002}*{Definition 4.9}, meaning that
  \begin{enumerate}
    \item \label{rr0} it has real rank \(0\);
    \item \label{inj-k0} for all \(p, q \in {\rm Proj}(B \tensor \CK)\), if \([p]_0 = [q]_0\) then \(p \oplus 1_{B} \sim q \oplus 1_{B}\);
    \item \label{surj-k1} the canonical map \(\CU_1(B) \to \Kgrp[1]{B}\) is surjective; and
    \item \label{surj-k0} the canonical map \({\rm Proj}(B \tensor \CK) \to \Kgrp[0]{B}\) is surjective.
  \end{enumerate}
  If every \(\{B_n\}_{n \in \NN}\) satisfies \cref{inj-k0} then the map \(\Kgrp[0]{\prod_{n \in \NN} B_n} \to \prod_{n \in \NN} \Kgrp[0]{B_n}\) is injective, since the equality \([p]_0 = [(p_n)_{n \in \NN}]_0 = [(q_n)_{n \in \NN}]_0 = [q]_0\) can be witnessed in \(\prod_{n \in \NN} M_{m(n) + 1}(B_n)\).
  Likewise, \cref{surj-k1,surj-k0} ensure that \(\phi \colon \Kgrp[\star]{\prod_{n \in \NN} B_n} \to \prod_{n \in \NN} \Kgrp[\star]{B_n}\) is surjective, as any element in the target can be realized by matrices of uniformly bounded size.
  Thus, all we have to do is show that \(\Kgrp[1]{\prod_{n \in \NN} B_n} \to \prod_{n \in \NN} \Kgrp[1]{B_n}\) is injective, which is done in \cite{dadarlat-eilers-2002}*{Lemma 3.3}, and crucially uses that \(B_n\) has real rank \(0\). 
  In the case when \(\{B_n\}_{n \in \NN}\) are simple, unital, infinite dimensional and tracially AF, it is proved in \cite{dadarlat-eilers-2002}*{Proposition 6.6} that they also satisfy \cref{rr0,inj-k0,surj-k1} as above.
  %By the same tokens as before this implies that the map \(\Kgrp[\star]{\prod_{n \in \NN} B_n} \to \prod_{n \in \NN} \Kgrp[\star]{B_n}\) is an isomorphism in \(\Kgrp[1]{\variable}\) and injective in \(\Kgrp[0]{\variable}\), as needed.
\end{proof}

\begin{proof}[Proof of \cref{main-theorem}]
  Fix separable \cstar{}algebras \(A\) and \(\{B_n\}_{n \in \NN}\) as in \cref{main-theorem}: \(A\) is separable and enjoys the UCT and \(\{B_n\}_{n \in \NN}\) is separable and the map \(\phi\) of \eqref{phi} is injective.
  Let \(\Phi\) be:
  \begin{equation} \label{eq:phi}
    \Phi \colon \KKgrp{A}{\prod_{n \in \NN} B_n} \to \prod_{n \in \NN} \KKgrp{A}{B_n}, \;\; \text{ where } \;
    \left[\left(\varphi_n\right)_{n \in \NN}, \left(\psi_n\right)_{n \in \NN}\right] \mapsto \left(\left[\varphi_n, \psi_n\right]\right)_{n \in \NN},
  \end{equation}
  where we use Cuntz' picture of \(\KKgrp{A}{B}\): elements in \(\KKgrp{A}{B}\) are homotopy classes of pairs of \Star{}homomorphisms \(\varphi, \psi \colon A \to \multialg{B \tensor \CK}\) such that \(\varphi(a) - \psi(a) \in B \tensor \CK\) for all \(a \in A\).
  Observe that \(\Phi\) is a surjective group homomorphism. Indeed, it is a group morphism because, coordinate wise, it is just given by the Kasparov product by the projection \(\pi_k \colon \prod_{n \in \NN} B_n \to B_k\):
  \[
    \KKgrp{A}{\prod_{n \in \NN} B_n} \cdot \left[\pi_k\right] \in \KKgrp{A}{\prod_{n \in \NN} B_n} \times \KKgrp{\prod_{n \in \NN} B_n}{B_k} \subseteq \KKgrp{A}{B_k},
  \]
  which indeed is a group morphism.

  We turn to injectivity of \(\Phi\).
  Noticing that \(\phi \colon \Kgrp[\star]{\prod_{n \in \NN} B_n} \to \prod_{n \in \NN} \Kgrp[\star]{B_n}\) is injective, one immediately realizes that \(\Phi\) fits in the following commuting diagram:
  \begin{equation} \label{diagram}
    \begin{tikzcd}
      \Extgrp{\Kgrp[\star]{A}}{\Kgrp[\star + 1]{\prod_{n \in \NN} B_n}} \arrow{d}{\Xi \phi^*} \arrow[hook]{r}{} & \KKgrp{A}{\prod_{n \in \NN} B_n} \arrow[two heads]{r}{} \arrow{d}{\Phi} & \Homgrp{\Kgrp[\star]{A}}{\Kgrp[\star]{\prod_{n \in \NN} B_n}} \arrow{d}{\Psi \phi^*} \\
      \prod_{n \in \NN} \Extgrp{\Kgrp[\star]{A}}{\Kgrp[\star+1]{B_n}} \arrow[hook]{r}{} & \prod_{n \in \NN} \KKgrp{A}{B_n} \arrow[two heads]{r}{} & \prod_{n \in \NN} \Homgrp{\Kgrp[\star]{A}}{\Kgrp[\star]{B_n}}.
    \end{tikzcd}
  \end{equation}
  The map on \(\Psi \phi^*\) on the right hand side is injective: it is precisely the map appearing in \cref{hom-iso} pre-composed with \(\phi \colon \Kgrp[\star]{\prod_{n \in \NN} B_n} \to \prod_{n \in \NN} \Kgrp[\star]{B_n}\), which is also injective by assumption.
  The horizontal rows are short exact because \(A\) enjoys the UCT.
  Thus, by the Five Lemma (or by chasing an element in \(\ker(\Phi)\) in \eqref{diagram}), in order to prove that \(\Phi\) is injective it suffices to prove that \(\Xi \phi^*\) is injective, where \(\phi^*\) is as in \eqref{eq-phi-induced}.
  The morphism \(\Xi\) is injective by \cref{ext-inj}, whence it suffices to prove that \(\phi^*\) is injective.
  By assumption, \(\phi\) is injective, so there is some short exact sequence
  \[
    0 \to \Kgrp[\star]{\prod_{n \in \NN} B_n} \xrightarrow{\phi} \prod_{n \in \NN} \Kgrp[\star]{B_n} \xrightarrow{q} H \to 0.
  \]
  Applying the long exact sequence of \(\Extgrp{\variable}{\variable}\) we find that
  \begin{align*}
    \ker\left(\Phi\right) & \subseteq \ker\left(\Xi'\right) \subseteq \ker\left(\Extgrp{\Kgrp[\star]{A}}{\Kgrp[\star]{\prod_{n \in \NN} B_n}} \xrightarrow{\phi^*} \Extgrp{\Kgrp[\star]{A}}{\prod_{n \in \NN} B_n}\right) \\
    & = {\rm im}\left(\Homgrp{\Kgrp[\star]{A}}{H} \to \Extgrp{\Kgrp[\star]{A}}{\Kgrp[\star]{\prod_{n \in \NN} B_n}}\right).
  \end{align*}
  Therefore, it suffices to prove that the right hand side above is zero.
  In the first scenario, \(\phi\) is an isomorphism, whence \(H = 0\). Thus \(\Homgrp{\Kgrp[\star]{A}}{H} = 0\) as well.
  In the second scenario, \(\Kgrp[\star]{A}\) is free and thus \(\Extgrp{\Kgrp[\star]{A}}{\variable} = 0\) is always trivial.
  The ``in particular'' statement is \cref{inj-kth}.
\end{proof}

\begin{proof}[Proof of \cref{intro-cor}]
  Going back to \eqref{diagram} one realizes that if \(\Psi \phi^*\) and \(\Xi \phi^*\) are isomorphisms then so is \(\Phi\).
  These are isomorphisms because so are \(\Psi\) and \(\Xi\), see \cref{hom-iso,ext-inj}, and so is \(\phi \colon \Kgrp[\star]{\prod_{n \in \NN} B_n} \to \prod_{n \in \NN} \Kgrp[\star]{B_n}\) by \cref{inj-kth}.
\end{proof}

As James Gabe pointed out in private communications, similar results to \cref{intro-cor} also follow from Gabe's work in \cite{gabe-memoirs-2024}.
The author is grateful for the short proof of the following, for instance.
\begin{corollary} \label[corollary]{cor-kl}
  Suppose \(A\) is \KKth-equivalent to an exact, separable \cstar{}algebra; and assume that \(\{B_n\}_{n \in \NN}\) are simple, unital and purely infinite. The following assertions hold.
  \begin{enumerate}
    \item \label{cor-kl-i} The map \(\KKgrp[nuc]{A}{\prod_{n \in \NN} B_n} \to \prod_{n \in \NN} \KKgrp[nuc]{A}{B_n}\) is surjective; and
    \item \label{cor-kl-ii} The map \(\KLgrp[nuc]{A}{\prod_{n \in \NN} B_n} \to \prod_{n \in \NN} \KLgrp[nuc]{A}{B_n}\) is an isomorphism.
  \end{enumerate}
  Moreover, should \(A\) be \KKth-equivalent to a nuclear \cstar{}algebra, then the same holds without the ``nuc'' subscripts.
\end{corollary}
\begin{proof}
  First of all, without loss of generality we may assume that \(A\) itself is separable and exact.
  By \cite{kirchberg-rordam-2002}*{Proposition 6.3 and Theorem 6.8}, \(B_k\) and \(\prod_{n \in \NN} B_n\) are strongly purely infinite in the sense of \cite{kirchberg-rordam-2002}*{Definition 5.1}.

  By \cite{gabe-memoirs-2024}*{Theorem 9.7} it follows that every nuclear \Star{}homomorphism \(A \to B_k\) (resp.\ \(A \to \prod_{n \in \NN} B_n\)) is strongly \(\CO_{\infty}\)-stable.
  By the existence and uniqueness results \cite{gabe-memoirs-2024}*{Theorems A and B} it follows that elements in \(\KKgrp[nuc]{A}{B_k}\) (resp.\ \(\KKgrp[nuc]{A}{\prod_{n \in \NN} B_n}\)) are in bijective correspondence with asymptotic Murray-von-Neumann classes of nuclear \Star{}homomorphisms \(A \to B_k\) (resp.\ \(A \to \prod_{n \in \NN} B_n\)).
  Combining these two yields that any tuple \(([\kappa_n])_{n \in \NN} \in \prod_{n \in \NN} \KKgrp[nuc]{A}{B_n}\) can be represented by (classes of) nuclear \Star{}homomorphisms \(\varphi_k \colon A \to B_k\), and these coupled together to get an element \(\kappa \in \KKgrp[nuc]{A}{\prod_{n \in \NN} B_n}\).
  This exactly shows item \cref{cor-kl-i}.
  Surjectivity in item \cref{cor-kl-ii} follows from a similar argument, this time using \cite{gabe-memoirs-2024}*{Theorem 8.10}.
  For injectivity in item \cref{cor-kl-ii} notice that \cite{gabe-memoirs-2024}*{Theorem 8.10} proves that equality in KL is precisely \emph{approximate} Murray-von Neumann equivalence, as opposed to \emph{asymptotic}.
  In particular this means that the equivalence in \(\KLgrp[nuc]{A}{\prod_{n \in \NN} B_n}\) may be witnessed in \(\prod_{n \in \NN} \KLgrp[nuc]{A}{B_n}\); \emph{i.e.}\ the map in \cref{cor-kl-ii} is injective.
  
  The ``moreover'' statements follows simply from the fact that if \(A\) is \KKth-equivalent to a nuclear \cstar{}algebra then \(\KKgrp[nuc]{A}{\variable} \cong \KKgrp{A}{\variable}\); and similarly for \(\KLgrp[nuc]{A}{\variable}\).
\end{proof}

%%%%%%%%%%%%%%%%%%%%%%%%%%%%%%%%%%%%%%%%%%%%%%%%%%%%%%%%%%%%%%%%%%%%%%%
% BIBLIOGRAPHY
%%%%%%%%%%%%%%%%%%%%%%%%%%%%%%%%%%%%%%%%%%%%%%%%%%%%%%%%%%%%%%%%%%%%%%%
\bibliography{bibprodkk}

\end{document}